\documentclass[a4paper,12pt]{amsart}
\usepackage{latexsym}
\usepackage{amssymb}
\usepackage{amsthm}
\usepackage{amstext}
\usepackage{amsmath}
\usepackage{enumerate}
\usepackage[parfill]{parskip}
\usepackage{comment}
\usepackage[active]{srcltx}
\usepackage[all]{xy}
\usepackage{mathtools}
\usepackage[a4paper]{geometry}

\DeclareMathOperator\soc{Soc}
\DeclareMathOperator\Top{Top}

\DeclareMathOperator\Hom{Hom}
\DeclareMathOperator\Ext{Ext}

\DeclareMathOperator\im{Im}
\DeclareMathOperator\cone{Cone}

\DeclareMathOperator\id{id}
\newtheorem{mythm}{Theorem}
\newtheorem{myex}[mythm]{Example}
\newtheorem{myrem}[mythm]{Remark}
\newtheorem{mylem}[mythm]{Lemma}
\newtheorem{myprop}[mythm]{Proposition}
\newtheorem{mycor}[mythm]{Corollary}

\numberwithin{equation}{section}
\newcommand{\Z}{\mathbb{Z}}

\newcommand{\cat}{\mathcal{C}}
\newcommand{\der}{\kom(\cgmod)}

\newcommand{\kom}{\text{Com}}
\newcommand{\gvsp}{\Bbbk\text{-gMod}}
\newcommand{\cgmod}{\cat\text{-gMod}}
\newcommand{\I}{V}
\newcommand{\J}{W}

\begin{document}
\title{Koszulity of some path categories}
\author[B.~Dubsky]{Brendan Dubsky}
\date{\today}
\begin{abstract}
We prove Koszulity of certain linear path categories obtained from connected graphs with some infinite directed walk. These categories can be viewed as locally quadratic dual to preprojective algebras. 
\end{abstract}

\maketitle

\section{Introduction}\label{s1}
Koszulity of graded (finite-dimensional, unital and associative) algebras has been extensively studied as a graded algebra analogue of semi-simplicity, cf. \cite{Pr70,Ke94,BGS96}. Koszul theory can be generalized to the setting of graded $\Bbbk$-linear categories, which may be regarded as associative but not necessarily finite-dimensional or unital algebras, cf. \cite{MOS09}. 

In the present article, we prove Koszulity of a certain class of linear path categories of connected graphs which contain some infinite directed walk, where the morphism relations are ``locally quadratic dual'' to those of preprojective algebras. In the special case of finite graphs, the resulting categories may thus be viewed as certain Koszul algebras of the traditional kind, while for graphs with vertices of infinite valency the categories will fall just outside of the setup of \cite{MOS09}.

The original motivation for the study of these categories was a desire to show Koszulity of blocks of modules over Lie superalgebras, which appeared in \cite{Ma14}. Indeed, our categories are generalizations of the categories of these blocks. Another example of an infinite-dimensional category of this kind is found in the shape of a generalized Khovanov's diagram algebra in \cite{BS12}. 

To prove Koszulity we recursively construct projective resolutions --- that are linear up to increasingly high position --- of the simple modules over our categories. The recursive construction is based on three classes of short exact sequences of resolutions, and corresponding mapping cones. 

In Section 2 we define our categories of interest, as well as their grading and graded modules. We then state a few previously known general results on categories with this kind of grading, including an inexplicit characterization of indecomposable projectives. In Section 3 we describe more explicitly these modules, as well as a few other useful ones. These latter modules will turn out to be Koszul in Section 4, where we will use them to construct our resolutions of simples and finally prove that these are increasingly linear. 

\subsection*{Acknowledgements}
The author is very grateful to his supervisor Volodymyr Mazorchuk for a lot of stimulating discussion and help with putting this article together. 

\section{Preliminaries}\label{s2}

\subsection{The path category $\cat$}

Throughout the text we will consider some fixed simple connected graph $G$ which contains some infinite directed walk, i. e. which has an orientation such that there is an infinite sequence of edges $e_1,e_2,\dots$ fulfilling that for any $n\in\Z_{>0}$ the target vertex of $e_{n}$ is the source vertex of $e_{n+1}$. We will denote by $\I$ the vertex set of $G$, and for each $i\in \I$ by $\I_i$ the set of vertices adjacent to $i$. These conditions will be satisfied by any infinite tree, as well as any simple connected graph containing a cycle.

We obtain our category of interest from $G$ through the following steps. First let $Q$ be the quiver obtained by replacing each edge of $G$ with one arrow in each direction. Let $\mathcal{C}''$ be the path category of $Q$, i. e. the objects of $\cat''$ are the vertices of $Q$ and the morphisms are paths in $Q$ (including an empty path at each object to serve as an identity) with composition given by concatenation. Next, fix some algebraically closed field, $\Bbbk$, and let $\mathcal{C}'$ be the $\Bbbk$-linear category generated by $\mathcal{C}''$, i.e. the $\Bbbk$-linear category with the same objects as $\mathcal{C}''$, the morphism spaces of $\mathcal{C}''$ constituting $\Bbbk$-vector space bases of the morphism spaces of $\mathcal{C}'$ and composition is defined by bilinearity and composition in $\cat''$ of the basis elements. Finally define $\mathcal{C}=\mathcal{C}'/\mathcal{I}$, where $\mathcal{I}$ is the ideal of $\mathcal{C}'$ generated by the following relations, which will henceforth be referred to as the first and second kind of relations on $\cat$ respectively. 
\begin{enumerate}[$($i$)$]
\item
For arrows $b:i\rightarrow j$ and $a:j\rightarrow k$ such that $i\ne j$, $j\ne k$ and $k\ne i$

\begin{equation*}
\xymatrix{
i\ar[r]^b&j\ar[r]^a&k
}
\end{equation*}
set $ab= 0$.

\item
For arrows $a_1:i\rightarrow j$, $b_1:j\rightarrow i$, $b_2:j\rightarrow k$ and $a_2:k\rightarrow j$ such that $i\ne j$, $j\ne k$ and $k\ne i$

\begin{equation*}
\xymatrix{
i\ar@<.5ex>[r]^{a_1}&j\ar@<.5ex>[l]^{b_1}\ar@<-.5ex>[r]_{b_2}&k\ar@<-.5ex>[l]_{a_2}
}
\end{equation*}
set $a_1b_1= a_2b_2$. 

\end{enumerate}

\subsection{A grading on $\cat$, and its graded modules}

This subsection will largely follow \cite[Section 2]{MOS09}, to which we refer the reader for proofs and full detail.

The category $\mathcal{C}$ has a $\Z_{\ge 0}$-grading given by the decomposition 
\begin{equation*}
\cat(i,j)=\bigoplus_{n\in\Z}\cat(i,j)_n
\end{equation*}
of each morphism space $\cat(i,j)$, where the degree $n$ part $\cat(i,j)_n$ is the subspace spanned by the (images under the quotient by $\mathcal{I}$ of) $Q$-paths of length $n$. 
A graded module of $\cat$ is a $\Bbbk$-linear functor $F:\cat\rightarrow \gvsp$ from $\cat$ to the category of all $\Z$-graded $\Bbbk$-vector spaces, which preserves the degrees of morphisms. We denote by $\cgmod$ the category of all graded $\cat$-modules. 

Let $n$ range over $\Z$ and $i$ over the objects of $\cat$. We define the \emph{degree shift} functors 
\begin{equation*}
\langle n\rangle:\cgmod\rightarrow \cgmod
\end{equation*}
on objects $M\in\cgmod$ by 
\begin{equation*}
M\langle n\rangle(i)_m=M(i)_{n+m}
\end{equation*}
for all $m\in\Z$ (and trivially on morphisms). 

Now, because we allow for infinite $\I_i$, the category $\cgmod$ falls just short of being positively graded in the sense of \cite[Definition 1]{MOS09}, but the proofs of \cite[Lemmas 3, 4, 5]{MOS09} still go through, and we have the following results.
\begin{enumerate}[$($i$)$]
\item
The category $\cgmod$ is abelian. 

\item
The module $P_i=\cat(i,\_)$ is an indecomposable projective object of $\cgmod$. 

\item
The degree shifted tops $L_i\langle n\rangle$ of $P_i$ constitute the simple objects of $\cgmod$ up to isomorphism. 
\end{enumerate}
In fact, it is easily seen that $L_i\cong P_i/\cat(i,\_)_{>0}$ is the one-dimensional module which is annihilated by all of $\cat$ except $\Bbbk\id_i$, where $\id_i$ is the identity morphism at $i$. 

\subsection{Koszul categories and resolutions}

In analogy with the definition for the case of finite-dimensional unital graded algebras, we define a $\Z_{\ge 0}$-graded category to be \emph{Koszul} if for all simple graded modules $L$ that are concentrated in degree 0 we have
\begin{equation}\label{koszdef}
\Ext^n_{\cgmod}(L,L\langle -m\rangle)\ne 0\Longrightarrow m=n.
\end{equation}

We denote by $\der$ the category of (cochain) complexes in $\cgmod$. For an object $\mathbf{X}^\bullet$ of $\der$, we by $\mathbf{X}^n$ mean the module at position $n$. For $X\in\cgmod$, we denote by $\mathbf{X}^\bullet$ the complex defined by $\mathbf{X}^n=0$ for $n\ne 0$ while $\mathbf{X}^0=X$. 

By abuse of notation, we will by $\langle n\rangle$ also denote the endofunctors on $\der$ which act on objects by applying $\langle n\rangle:\cgmod\rightarrow\cgmod$ to the modules and morphisms in the complexes (and which act trivially on morphisms in $\der$). 

We may now formulate a sufficient (and in fact necessary, though we will not pursue this) criterion for Koszulity.
\begin{myprop}
\label{linres}
A $\Z_{\ge 0}$-graded category $\mathcal{B}$ is Koszul if for every $N\in\Z$, each simple graded $\mathcal{B}$-module $L$ concentrated in degree 0 has a graded projective resolution $\mathbf{P}^\bullet\in\der$ such that for all $n\ge N$, the module $\mathbf{P}^n$ is generated by $(\mathbf{P}^n)_n$. 
\end{myprop}
\begin{proof}
Fix $N$ and a corresponding $\mathbf{P}^\bullet$ satisfying the assumptions of the proposition. By definition 
\begin{equation*}
\Ext^n_\cat(L,L\langle -m\rangle)=H^n(\Hom_{\cgmod}(\mathbf{P}^\bullet,L\langle -m\rangle)).
\end{equation*}
 But $\Hom_{\cgmod}(\mathbf{P}^n,L\langle -m\rangle)=0$ unless $m=n$, since $\mathbf{P}^n$ is generated in degree $n$, and the desired result follows. 
\end{proof}
We will later utilize the mapping cone (cf. e.g. \cite[p. 154]{GM03}) of a morphism of complexes $f: \mathbf{X}^\bullet\rightarrow \mathbf{Y}^\bullet$, which we shall denote by $\cone(f)$.

\section{Some useful $\cat$-modules described}

From the following lemma one readily sees that no indecomposable projective module in $\cgmod$ has graded length greater than 2.

\begin{mylem}
\label{shortlemma}
For any $i,j\in\cat$ and $n>2$ we have $\cat(i,j)_n=0$. 
\end{mylem}
\begin{proof}
Since $\cat(i,j)$ is generated by degree 1 morphisms, it suffices to show that $\cat(i,j)_3=0$. From the first kind of relations on $\cat$, we see that the only morphisms in $\cat(i,j)_2$ are linear combinations of morphisms of the form $ab$, where $a:i\rightarrow j$ and $b:j\rightarrow i$ are some arrows in $Q$, and consequently any morphism in $\cat(i,j)_3=0$ has to be a linear combination of morphisms of the form $aba$. 

Of course $|\I|>2$, so at least one of $i$ and $j$ is adjacent to more than one vertex. Both $ab$ and $ba$ are loops, so therefore the second kind of relations on $\cat$ gives that either $aba=cda$ or $aba=acd$, for some other loop $cd$ in $Q$. But in either case the first kind of relations again gives that $aba=0$. 
\end{proof}

For the rest of this text we will denote by $i$ some arbitrary element of $\I$, and by $\J$ some arbitrary subset of $\I_i$.

\subsection{A closer look at the indecomposable projective $\cat$-modules}
The following lemma describes the module $P_i$ more explicitly.
\begin{myprop}
\label{projprop}
The module $P_i$ satisfies the following properties.
\begin{enumerate}[$($i$)$]
\item
$\Top(P_i)\cong L_i$.

\item
$\soc(P_i)\cong L_i\langle -2\rangle$. 

\item
$\soc(P_i/\soc(P_i))\cong \bigoplus_{k\in I_i}L_k\langle -1\rangle$.
\end{enumerate}
\end{myprop}
\begin{proof}
The first claim holds by the definition of $L_i$. 

By Lemma \ref{shortlemma}, we have $(P_i)_{\ge 3}= 0$. From the first kind of relations on $\cat$ we get that $(P_i)_2\subset \cat(i,i)$. Since it must be annihilated by morphisms of positive degree, $(P_i)_2$ is in fact annihilated by all of $\cat$ except for $\Bbbk\id_i$, and hence must be a sum of copies of $L_i\langle -2\rangle$. From the second kind of relations it follows that there is only one summand in $(P_i)_2$, hence $\soc(P_i)\cong L_i\langle -2\rangle$. 

From the second kind of relations now also follows that the socle has no summand in degree 1, since otherwise all degree 2 morphisms of $\cat$ would kill $P_i$. Finally, $\soc(P_i/\soc(P_i))\cong \cat(i,\_)_1/\cat(i,\_)_2\cong \bigoplus_{k\in I_i}L_k\langle -1\rangle$. 
\end{proof}
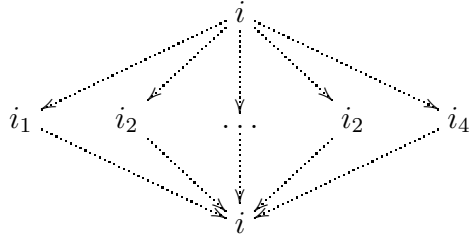
\begin{figure}
\[
\xymatrix{
&&i\ar@{.>}[dll]\ar@{.>}[dl]\ar@{.>}[d]\ar@{.>}[dr]\ar@{.>}[drr]
&&\\
i_1\ar@{.>}[drr]&i_2\ar@{.>}[dr]&\dots\ar@{.>}[d]&i_2\ar@{.>}[dl]&i_4\ar@{.>}[dll]\\
&&i&&
}
\]
\caption{Illustration of the module $P_i$ where we have $\{i_1,i_2,i_3,i_4,\dots\}= \I_i$. The layers of indices represent the layers of corresponding simple modules in the radical filtration of $P_i$. Thus if the index $j$ lies in the layer above an index $k$, the module $P_i$ has a subquotient whose simple subquotients are degree shifts of $L_j$ and $L_k$. If there is an arrow from $j$ to $k$, then the subquotient has (up to degree shifts) top $L_j$ and socle $L_k$, whereas otherwise it is semisimple. }\label{projfig}
\end{figure}
\begin{myrem}
{
\rm
We observe in passing that $\cat\cong\cat^\text{op}$, and that this self-duality preserves the grading of $\cat$. It follows that the self-duality of $\cgmod$ induced by $\Hom_{\Bbbk}(\_,\Bbbk)$ takes $P_i$ to itself, so that in fact $P_i=I_i$, where $I_i$ is the injective envelope of $L_i$.
}
\end{myrem}

\subsection{The modules $M(\J,i)$}

From the simple and projective modules we obtain another family of modules that will be important to us.
\begin{enumerate}[$($i$)$]
\item
By $M(\I_i,i)$ we denote the kernel of the natural surjection 
\begin{equation*}
P_i\langle 1\rangle\twoheadrightarrow L_i\langle 1\rangle.
\end{equation*}

\item
More generally, we denote by $M(\J,i)$ the submodule of $M(\I_i,i)$ whose top has as summands precisely those $L_k$ for which $k\in \J$ in case $\J$ is non-empty, and the module $L_i\langle -1\rangle$ otherwise. 
\end{enumerate}

\begin{mycor}
\label{mcor}
The module $M(\J,i)$ is isomorphic to any module $M\in\cgmod$ which satisfies the following properties. 
\begin{enumerate}[$($i$)$]
\item
$\Top(M)\cong \bigoplus_{k\in\J}L_k$.

\item
$\soc(M)\cong L_i\langle{-1}\rangle$. 
\end{enumerate}
\end{mycor}
\begin{proof}
From Proposition \ref{projprop} and the subsequent definition of $M(\J,i)$ follows that $M(\J,i)$ satisfies the given properties. 

The maps of basis elements 
\begin{equation*}
\soc(M(\J,i))\cong L_i\langle-1\rangle\ni v\mapsto w\in L_i\langle-1\rangle\cong \soc(M)
\end{equation*}
and 
\begin{equation*}
\Top(M(\J,i))\supset L_k\ni v_k\mapsto w_k\in L_k\subset \Top(M)
\end{equation*}
chosen such that $a_kv_k=v$ and $a_kw_k=w$ for basis elements $a_k\in\cat(k,i)$ extends by linearity to an isomorphism. 
\end{proof}

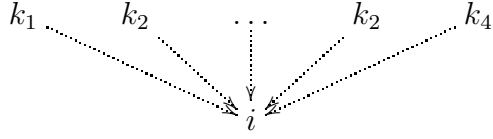
\begin{figure}
\[
\xymatrix{
k_1\ar@{.>}[drr]&k_2\ar@{.>}[dr]&\dots\ar@{.>}[d]&k_2\ar@{.>}[dl]&k_4\ar@{.>}[dll]\\
&&i&&
}
\]
\caption{Illustration of the module $M(\J,i)$ for a non-empty vertex set $\J=\{k_1,k_2,k_3,k_4,\dots\}$.}\label{modfig}
\end{figure}

\section{$\cat$ is Koszul}\label{s3}

Our main result reads as follows.

\begin{mythm}
\label{koszthm}
The category $\mathcal{C}$ is Koszul. 
\end{mythm}

\subsection{Three families of short exact sequences}
The recursive construction of increasingly linear resolutions will be based on three short exact sequences, the injections of which we here define. 

Assume in this subsection that $\J$ contains a vertex $j$ such that there is an infinite directed walk in $G$ starting with the directed edge from $i$ to $j$. 

Let the maps 

\begin{equation}\label{alphadef}
\alpha_i:M(\I_i,i)\langle -1\rangle \hookrightarrow P_i
\end{equation}
and

\begin{equation}\label{betadef}
\beta_{i,j}:M(\I_{j}\backslash\{i\},j)\langle -1\rangle\hookrightarrow P_{j}
\end{equation}

 be (the degree shifted versions of) the natural inclusions which define $M(\I_i,i)$ and $M(\I_{j}\backslash\{i\},j)$. 

To define the third family of maps, we first define the following auxiliary maps. For $k\in \J$, let 
\begin{equation}\label{aux1}
\iota_k:L_i\langle -1\rangle\hookrightarrow M(\{k\},i)
\end{equation}
be the natural inclusion, and let 

\begin{equation}\label{aux2}
\varphi_k:M(\I_k,k)\langle -1\rangle\twoheadrightarrow L_i\langle -1\rangle
\end{equation}
be the quotient map. 

Now define
\begin{equation}\label{gammadef1}
\gamma_{\J,j,i}: \bigoplus_{k\in \J\backslash\{j\}}M(\I_k,k)\langle -1\rangle\rightarrow (\bigoplus_{k\in \J\backslash\{j\}} P_k)\oplus M(\{j\},i)
\end{equation}
to be given by the $\J\times \J\backslash\{j\}$-matrix 
\begin{equation}\label{gammadef2}
(\gamma_{\J,j,i})_{r,c}=\begin{cases}
\alpha_{r}, &\text{if $r=c\in \J\backslash\{j\}$};\\
-\iota_{j}\circ\varphi_c, &\text{if $c\in \J\backslash\{j\}$ and $r=j$};\\
0, &\text{otherwise.}
\end{cases}
\end{equation}
\begin{mylem}
\label{sesdef}
We have the following short exact sequences of modules.

\begin{enumerate}[$($i$)$]
\item
\begin{equation*}
M(\I_i,i)\langle -1\rangle \overset{\alpha_i}\hookrightarrow P_i\twoheadrightarrow L_i,
\end{equation*}

\item
\begin{equation*}
M(\I_{j}\backslash\{i\},j)\langle -1\rangle\overset{\beta_{i,j}}\hookrightarrow P_{j}\twoheadrightarrow M(\{j\},i),
\end{equation*}

\item
\begin{equation*}
\bigoplus_{k\in \J\backslash\{j\}}M(\I_k,k)\langle -1\rangle \overset{\gamma_{\J,j,i}}\hookrightarrow (\bigoplus_{k\in \J\backslash\{j\}} P_k)\oplus M(\{j\},i)\overset{\delta_{\J,j,i}}\twoheadrightarrow M(\J,i).
\end{equation*}

\end{enumerate}

\end{mylem}
\begin{proof}
The existence of the first exact sequence follows from the definition of $M(\I_i,i)$. This proves claim (i). 

By Proposition \ref{projprop} we get that $P_j/\im(\beta_{i,j})$ has top $L_j$ and socle $L_i$, hence by Corollary \ref{mcor} is isomorphic to $M(\{j\},i)$. This proves claim (ii). 

Now consider the quotient map 
\begin{equation*}
\delta_{\J,j,i}: (\bigoplus_{k\in \J\backslash\{j\}} P_k)\oplus M(\{j\},i)\twoheadrightarrow (\bigoplus_{k\in \J\backslash\{j\}} P_k)\oplus M(\{j\},i)/\im(\gamma_{\J,j,i}). 
\end{equation*}
Because each $\varphi_c$ annihilates every $M(\I_k\backslash\{i\},k)\subset M(\I_k,k)$, we see from the definition of $\gamma_{\J,j,i}$ that 
\begin{equation*}
(\bigoplus_{k\in \J\backslash\{j\}} M(\I_k\backslash\{i\},k)\langle -1\rangle)\oplus 0\subset \im(\gamma_{\J,j,i}),
\end{equation*}
so that $\delta_{\J,j,i}$ factors through 
\begin{equation*}
\bigoplus_{k\in \J\backslash\{j\}} P_k\oplus M(\{j\},i)/(\bigoplus_{k\in \J\backslash\{j\}} M(\I_k\backslash\{i\},k)\langle -1\rangle\oplus 0)\cong \bigoplus_{k\in \J}M(\{k\},i),
\end{equation*}
also using Corollary \ref{mcor} as in the case of the second short exact sequence. 

Because $(\im(\gamma_{\J,j,i}))_0=0$ it is clear that $\Top(\im(\delta_{\J,j,i}))\cong\bigoplus_{k\in\J}L_k$. Finally $\delta_{\J,j,i}$ enforces the relations of the form $v_k=v_{j}$, where each $v_k$ is some basis element of $\iota_{k}(L_i\langle-1\rangle)\subset \bigoplus_{k\in \J}M(\{k\},i)$, so that the summands of $\soc(\bigoplus_{k\in \J}M(\{k\},i))$ are identified. Then $\soc(\im(\delta_{\J,j,i}))\cong L_i$, so that again by Corollary \ref{mcor} we get the sought $\im(\delta_{\J,j,i})\cong M(\J,i)$. This proves claim (iii). 
\end{proof}

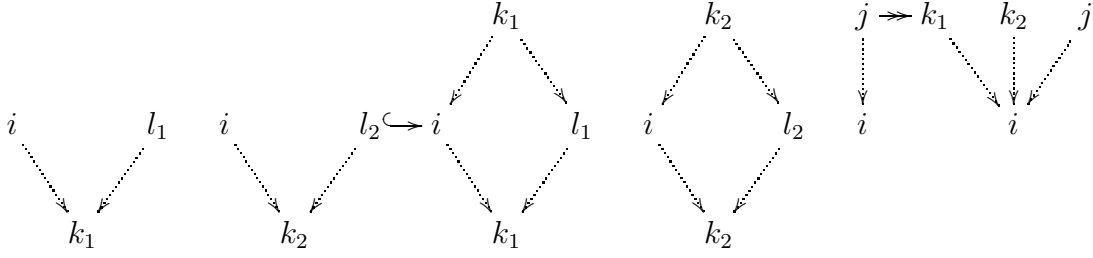
\begin{figure}
\[
\xymatrix@C=1em{
&&&&&&&k_1\ar@{.>}[dl]\ar@{.>}[dr]&&&k_2\ar@{.>}[dl]\ar@{.>}[dr]&&j\ar@{.>}[d]\ar@{->>}[r]&k_1\ar@{.>}[dr]&k_2\ar@{.>}[d]&j\ar@{.>}[dl]\\
i\ar@{.>}[dr]&&l_1\ar@{.>}[dl]&i\ar@{.>}[dr]&&l_2\ar@{.>}[dl]\ar@{^{(}->}[r]&i\ar@{.>}[dr]&&l_1\ar@{.>}[dl]&i\ar@{.>}[dr]&&l_2\ar@{.>}[dl]&i&&i&\\
&k_1&&&k_2&&&k_1&&&k_2&&&&&
}
\]
\caption{Illustration of the third short exact sequence for the case where $\J=\{k_1,k_2,j\}$ and $|\I_{k_1}|=|\I_{k_2}|=2$.}\label{ses3fig}
\end{figure}

For the recursive construction to come we want that every summand in the domains and codomains of the $\alpha_i$, $\beta_{i,j}$ and $\gamma_{\J,j,i}$ are either projective or, up to a degree shift, of the same form as the rightmost modules of the short exact sequences from Lemma \ref{sesdef}, i.e. of the form $L_i$, $M(\{j\},i)$ or $M(\J,i)$, for some choices of $i,\J$ and $j$, possibly different from those before. 

Obviously each $M(\I_i,i)$ is of the form $M(\J,i)$ with some vertex fulfilling the role of $j$ in $\J$. That also $M(\I_{j}\backslash\{i\},j)$ is of this form is a consequence of the following trivial lemma. 

\begin{mylem}\label{noleaf}
There is some $k\in \I_{j}\backslash\{i\}$ such that there is an infinite directed walk in $G$ starting with the directed edge from $j$ to $k$. 
\end{mylem}
\begin{proof}
Take $k$ to be the second vertex along an infinite directed walk in $G$ starting with the directed edge from $i$ to $j$.
\end{proof}

\subsection{Recursive construction of resolutions}
We may now recursively define for all choices of $i$, $\J$ and $j$ sequences $\mathbf{A}_n(i)^\bullet$, $\mathbf{B}_n(i,j)^\bullet$ and $\mathbf{C}_n(\J,i)^\bullet$, indexed by $n\in\Z_{\ge 0}$, of projective resolutions as follows (these are well-defined by Lemma \ref{noleaf} and the discussion preceding it). 
\begin{enumerate}[$($i$)$]
\item
\begin{enumerate}[$($a$)$]
\item
Let $\mathbf{A}_0(i)^\bullet$ be some projective resolution of $L_i$.

\item
Let $\mathbf{B}_0(i,j)^\bullet$ be some projective resolution of $M(\{j\},i)$.

\item
Let $\mathbf{C}_0(\J,i)^\bullet$ be some projective resolution of $M(\J,i)$. 

\end{enumerate}
We will only assume that each of these resolutions is generated in degree $0$ at position $0$, which is possible since $L_i$, $M(\{j\},i)$ and $M(\J,i)$ are generated in degree $0$.

\item
Let $n>0$, and let

\begin{equation*}
\alpha_i^n:\mathbf{C}_{n-1}(\I_i,i)^\bullet\langle -1\rangle\rightarrow \mathbf{P}_i^\bullet
\end{equation*}

define some morphism of complexes $\alpha_i^\bullet$ induced by $\alpha_i$. Set
\begin{equation*}
\mathbf{A}_n(i)^\bullet=\cone(\alpha_i^\bullet).
\end{equation*}

Similarly, let 
\begin{equation*}
\beta_{i,j}^n:\mathbf{C}_{n-1}(\I_{j}\backslash\{i\},j)^\bullet\langle -1\rangle\rightarrow \mathbf{P}_{j}^\bullet
\end{equation*}
define some morphism of complexes $\beta_{i,j}^\bullet$ induced by $\beta_{i,j}$. Set 
\begin{equation*}
\mathbf{B}_n(i,j)^\bullet=\cone(\beta_{i,j}^\bullet).
\end{equation*}

Finally let
\begin{equation*}
\gamma_{\J,j,i}^n:\bigoplus_{k\in \J\backslash\{j\}}\mathbf{C}_{n-1}(\I_k,k)^\bullet \langle -1\rangle\rightarrow(\bigoplus_{k\in \J\backslash\{j\}}\mathbf{P}_k^\bullet) \oplus \mathbf{B}_n(i,j)^\bullet 
\end{equation*}
define some morphism of complexes $\gamma_{\J,j,i}^n$ induced by $\gamma_{\J,j,i}$. Set
\begin{equation*}
\mathbf{C}_n(\J,i)^\bullet=\cone(\gamma_{\J,j,i}^\bullet)
\end{equation*}
(this does not depend on $j\in \J$). 
\end{enumerate}

We may now prove Theorem \ref{koszthm}. 
\begin{proof}
First note that each $\mathbf{A}_n(i)^\bullet$, $\mathbf{B}_n(i,j)^\bullet$ and $\mathbf{C}_n(\J,i)^\bullet$ is a projective resolution of $L_i$, $M(\{j\},i)$ and $M(\J,i)$, respectively, because of the short exact sequences of Lemma \ref{sesdef}. 

Thanks to Proposition \ref{linres}, it suffices to show that, for all $n\in\Z_{\ge 0}$, the above resolutions are linear for positions at least $-n$, i. e. that for $m\le n$ each resolution is generated in degree $m$ at position $-m$. This we show by induction on $n$. 

The base case is clear by definition of $\mathbf{A}_0(i)^\bullet$, $\mathbf{B}_0(i,j)^\bullet$ and $\mathbf{C}_0(\J,i)^\bullet$. 

As for the induction step, assume that the statement is true for some $n\in\Z_{\ge 0}$. Then for $m\le n+1$ we have that $\mathbf{A}_{n+1}(i)^{-m}=\mathbf{C}_{n}(\I_i,i)^{-(m-1)}\langle -1\rangle\oplus \mathbf{P}_i^{-m}$ is generated in degree $m$, because by the induction assumption $\mathbf{C}_{n}(\I_i,i)^{-(m-1)}\langle -1\rangle$ is generated in degree $m-1-(-1)=m$. Thus the statement is true for $n+1$ as well. 

The case of the sequence of $\mathbf{B}_n(i,j)^\bullet$ is entirely analogous, and using this case and the fact that 
\begin{equation*}
\mathbf{C}_{n+1}(\J,i)^{-m}=\bigoplus_{k\in \J\backslash\{j\}}\mathbf{C}_{n}(\I_k,k)^{-(m-1)} \langle -1\rangle\oplus (\bigoplus_{k\in \J\backslash\{j\}}\mathbf{P}_k^{-m}) \oplus \mathbf{B}_{n+1}(i,j)^{-m} 
\end{equation*}
the last case is in the same way taken care of.

\end{proof}

\begin{myrem}
{\rm
In the special case of $G$ being a tree containing a finite non-Dynkin connected subtree it has been suggested by Volodymyr Mazorchuk that the property of $\cat$ being ``locally dual'' to a preprojective algebra could be exploited to give an alternative proof of Theorem \ref{koszthm} along the following lines. Consider the finite tree obtained by removing all vertices of distance greater than $n$ from a given vertex $i$ of $G$. The relations at vertices of distance no greater than $n-1$ from $i$ will be quadratic dual to those of a preprojective algebra. By \cite[Theorem 3.4.2]{EE07}, such an algebra will be Koszul, hence its quadratic dual is as well. Then $L_i$ has a resolution which is linear for positions at least $-(n-1)$. Now let $n$ approach infinity. 
}
\end{myrem}

\subsection{A non-example}
Our assumption that $G$ is orientable such that it contains an infinite directed walk is essential. Indeed, if it does not it may happen that $V_j\backslash\{i\}=\varnothing$ in the second short exact sequence of Lemma \ref{sesdef}. But then $M(V_j\backslash\{i\},j)\langle-1\rangle=L_j\langle-2\rangle$ so that $\mathbf{B}_n(i,j)^{-m}$ will not be generated in degree $m$.

The following minimal concrete example illustrates what goes wrong.
\begin{myex}
{
\rm
\label{nonex}
Let $G$ be the graph
\begin{equation*}
\xymatrix{
1\ar@{-}[r]&2
}
\end{equation*}
and consider the following projective resolution of $L_1$.

\begin{equation*}
\xymatrix{
&&&1\ar@{.>}[d]\ar@{->>}[r]&1\\
&&2\ar@{.>}[d]\ar[r]&2\ar@{.>}[d]&\\
&&1\ar[r]\ar@{.>}[d]&1&\\
&2\ar@{.>}[d]\ar[r]&2&&\\
\dots\ar[r]&1\ar@{.>}[d]&&&\\
\dots\ar[r]&2&&&
}
\end{equation*}

Here the projective module at position $-2$ is indeed generated in degree 3 rather than 2, so the resolution is not linear there. 
}
\end{myex}

\end{document}